\documentclass[12pt]{article}

\usepackage[T1]{fontenc}
\usepackage{csquotes,enumerate,amssymb,amsmath}
\usepackage[backend=bibtex,firstinits=true, url=false, isbn=false]{biblatex}
\usepackage{amsmath}
\usepackage[utf8]{inputenc}
\usepackage[T1]{fontenc}
\usepackage{mathtools}

\def\theequation{\arabic{equation}}

  \def\theequation{\thesection.\arabic{equation}}
  \def\theequation{\arabic{equation}}

\newsavebox{\savepar}

\newtheorem{corollary}{Corollary}[section]
\newtheorem{remark}{Remark}[section]
\newtheorem{remarks}{Remarks}[section]

\newtheorem{theorem}{Theorem}[section]

\def\appendix{\par
   \setcounter{section}{0}
   \setcounter{equation}{0}
   \def\@chapapp{APPENDIX}
   \def\thesection{\Alph{section}}
   \def\theequation{A.\arabic{equation}}
   \def\section{
   \setcounter{equation}{0}
   \refstepcounter{section}
   \@startsection {section}{A}{0pt}{-3.5ex \@plus -1ex \@minus -.2ex}
                  {0.3ex \@plus.2ex}{\normalsize\sffamily\bfseries} }}

\def\eop{\hbox{{\vrule height7pt width3pt depth0pt}}}

\makeatletter
\newcommand{\least}{\let\CS=\@currsize\renewcommand{\baselinestretch}{1}\tiny\CS}
\newcommand{\singlespacing}{\renewcommand{\baselinestretch}{1.5}}
\newcommand{\oneandahalfspacing}{\let\CS=\@currsize\renewcommand{\baselinestretch}{1.2}\tiny\CS}
\newcommand{\doublespacing}{\let\CS=\@currsize\renewcommand{\baselinestretch}{2.5}\tiny\CS}

\usepackage{amsmath}
\usepackage{amsfonts}
\usepackage{euscript}
\usepackage{oldgerm}
\usepackage{rotating}

  \renewcommand{\baselinestretch}{1.3}

\makeindex

\newenvironment{proof}
{\begin{sloppypar}\noindent{\sf Proof.~}}
{\hspace*{\fill}\eop
\medskip
\end{sloppypar}}

\numberwithin{equation}{section}
\newcommand{\be}{\begin{equation}}
\newcommand{\ee}{\end{equation}}

\newcommand{\beano}{\begin{eqnarray*}}
\newcommand{\eeano}{\end{eqnarray*}}
\newcommand{\bea}{\begin{eqnarray}}
\newcommand{\eea}{\end{eqnarray}}

\newcommand{\ba}{\begin{array}}
\newcommand{\ea}{\end{array}}

\newcommand{\vone}{\vskip 2ex}
\newcommand{\vtwo}{\vskip 4ex}

\newcommand{\T} {{\mathbb{T}}}

\begin{document}
\singlespacing
%\doublespacing

\begin{center}
{\Large \bf Beta-Function Identities via Probabilistic Approach}\\
\end{center}

\vtwo
\begin{center}
{\large \bf  P. Vellaisamy\textsuperscript{1} and A. Zeleke \textsuperscript{2}} \\
 \textsuperscript{1}Department of Mathematics, Indian Institute of Technology Bombay \\
\noindent Powai, Mumbai-400076, India.\\
{Email: \it pv@math.iitb.ac.in}

 \textsuperscript{2}Lyman Briggs College \& Department of Statistics \& Probability, \\
\noindent Michigan State University. East Lansing, MI 48825, USA\\
{Email: \it zeleke@stt.msu.edu}
\end{center}

\abstract {Using a probabilistic approach, we derive several interesting identities involving beta functions. Our results generalize certain well-known combinatorial identities involving binomial coefficients and gamma functions.}

\vone
{\noindent \bf Keywords}. {Binomial inversion, combinatorial identities, gamma random variables, moments,  probabilistic proofs.}

\vone
\noindent  MSC2010 {\it Subject Classification}: Primary: 62E15, 05A19; Secondary: 60C05.\\

\section{Introduction}
There are several interesting combinatorial identities involving binomial coefficients, gamma functions, hypergeometric functions  (see, for example, Riordan (1968), Petkovsek et al (1996), Bagdasaryan (2015), Srinivasan (2007) and Vellaisamy and Zeleke (2017)), {\it etc}. One of these is the famous identity that involves the convolution of the central binomial coefficients: \begin{equation} \label{eqn1.1}
 \sum_{k=1}^{n} \binom{2k}{k}\binom{2n-2k}{n-k}=4^n.
\end{equation}   
In recent years, researchers have provided several proofs of $(1.1)$. 
A proof that uses generating functions can be found in Stanley (1997). The combinatorial proofs can also be found, for example, in Sved (1984), De Angelis (2006) and Miki{\' c} (2016).  A computer generated proof using the WZ method is given by Petkovsek, Wilf and Zeilberger(1996). Chang and  Xu (2011) extended the identity in \eqref{eqn1.1}
and presented a probabilistic proof of the identity
\begin{equation}\label{eqn1.2}
 \sum_{\substack{k_j \geq 0,\; 1 \leq j \leq m;\\ \sum_{J=1}^m  k_j=n}} \binom{2k_1}{k_1}\binom{2k_2}{k_2}\cdots\binom{2k_m}{k_m}=\frac{4^n}{n!} \frac{\Gamma(n+\frac{m}{2})}{\Gamma(\frac{m}{2})},
\end{equation}
 for positive integers $m$ and $n$, and  Miki{\' c} (2016) presented a combinatorial proof of this identity based on the method of recurrence relations and telescoping.
 
 \noindent A related identity for the alternating convolution of central binomial coefficients is 
 \begin{equation}\label{eqn1.3}
 \sum_{k=0}^{n} (-1)^k \binom{2k}{k}  \binom{2n-2k}{n-k} =
 \begin{cases*}
 2^n \binom{n}{\frac{n}{2}}, & if  $n$ is even  \\
 \phantom{0 }0, & if  $n$ is odd.
 \end{cases*} 
 \end{equation}
 
 \noindent The combinatorial proofs of the above identity can be found in, for example, Nagy (2012), Spivey (2014) and  Miki{\' c} (2016).
 Recently, Pathak (2017) has given a probabilistic of the above identity.
 
 \vspace*{.3cm}
 \noindent Unfortunately, in the literature, there are only a few identities that involve  beta functions are available.  Our goal in this paper is to establish several interesting beta-function identities,
 similar to the ones given in \eqref{eqn1.1} and \eqref{eqn1.3}. Interestingly, our results generalize all the above-mentioned identities, including the main result  \eqref{eqn1.2} of Chang and  Xu (2011). Our approach is based on probabilistic arguments, using the moments of the sum or the difference of tow gamma random variables.   

\section{Identities Involving the Beta Functions}
The beta function, also known as the Euler first integral, is defined by 
\begin{equation} \label{eqn2.1}
B(x,y)=\int_0^1  t^{x-1}(1-t)^{y-1}dt,\;  x, y>0.
\end{equation}
 It was studied by Euler and Legendre and is related to the gamma functions by 
\begin{equation}\label{eqn2.2}
B(x.y)=\frac{\Gamma(x)\Gamma(y)}{\Gamma(x+y)},
\end{equation}
where $\Gamma(x)=\int_0^\infty t^{x-1}e^{-t} dt, x>0.$ The beta function is symmetric, {\it i.e.} $B(x,y)=B(y,x)$ and satisfies the basic identity 
\begin{equation} \label{eqn2.3}
 B(x,y)=B(x,y+1)+B(x+1,y),~ \text{for}~x, y>0. 
 \end{equation}

\noindent  Using a  probabilistic approach, we generalize, in some sense, the above basic identity in \eqref{eqn2.3} in two different directions.

\noindent Let $ X$ be a gamma  random variable with parameter $p >0$, denoted by $ X \sim G(p)$, and density $$f(x|p)=\frac{1}{\Gamma(p)}e^{- x}x^{p-1},\;x>0,\; p >0.$$  Then, it follows  (see Rohatgi and Saleh (2002), p.~212) that
\begin{eqnarray} \label{eqn2.4}
E(X^n)=\frac{1}{\Gamma(p)} \int_0^\infty e^{- x}x^{p+n-1}dx=\frac{\Gamma(p+n)}{\Gamma(p)}. 
\end{eqnarray}
Let $X_1 \sim G(p_1)$ and $X_2 \sim G(p_2)$ be two independent gamma distributed  random variables, with parameters $p_1$ and $p_2$, respectively. Then it is known that $Y = X_1+X_2$ follows a gamma distribution with parameters $(p_1+p_2)$, {\it i.e}, $Y \sim G(p_1+p_2)$. 
We compute the $n$-th moment $E(Y^n)$ in two different ways, and equating them gives us an identity involving  beta functions.
\begin{theorem} Let $p_1,\;p_2 > 0$. Then for any  integer $n \geq 0$, 
\begin{equation}\label{eqn2.5}
\displaystyle \sum_{k=0}^n \binom{n}{k}B(p_1+k,p_2+n-k)=B(p_1,p_2).
\end{equation}
\end{theorem}

\begin{proof}
\noindent Since $ X_1+X_2=Y \sim G(p_1+p_2)$, we get,
 from  \eqref{eqn2.4},
\begin{equation} \label{eqn2.6}
E(Y^n)=\frac{\Gamma(p_1+p_2+n)}{\Gamma(p_1+p_2)}. 
\end{equation}

\noindent Alternatively, since $X_1 \sim G(p_1),\; X_2 \sim G(p_2)$ and are independent, we have by using the binomial theorem  
\begin{eqnarray}
E(Y^n)&=&E(X_1+X_2)^n=E\Big(\sum_{k=0}^n \binom{n}{k} X_1^kX_2^{n-k}\Big)\nonumber\\&=& \sum_{k=0}^n \binom{n}{k} E(X_1^k) E(X_2^{n-k})\nonumber\\
&=& \sum_{k=0}^n \binom{n}{k}\frac{\Gamma(p_1+k) \Gamma(p_2+n-k)}{\Gamma(p_1)\Gamma(p_2)} \label{eqn2.7}, 
\end{eqnarray}
using \eqref{eqn2.4}. 
 
\noindent Equating \eqref{eqn2.6} and \eqref{eqn2.7} , we obtain       

$$\frac{\Gamma(p_1+p_2+n)}{\Gamma(p_1+p_2)}=\sum_{k=0}^n \binom{n}{k}\frac{\Gamma(p_1+k) \Gamma(p_2+n-k)}{\Gamma(p_1)\Gamma(p_2)}$$
which leads to
$$\sum_{k=0}^n \binom{n}{k}\frac{\Gamma(p_1+k) \Gamma(p_2+n-k)}{\Gamma(p_1+p_2+n)}=\frac{\Gamma(p_1)\Gamma(p_2)}{\Gamma(p_1+p_2)}.$$
This proves the result.
\end{proof}

\begin{remarks} {\em 
\begin{enumerate} 
\item[(i)] Indeed, one could also consider $X_1 \sim G(\alpha, p_1)$ and $X_2 \sim G(\alpha, p_2)$, $\alpha >0$, the two-parameter gamma random variables. But, the result does not change as  as the powers of $\alpha$ cancel out.    
\item[(ii)] When $p_1=p_2=\frac{1}{2}$, we get, $$\sum_{k=0}^n \binom{n}{k}B\Big(\frac{1}{2}+k, \frac{1}{2}+n-k\Big) =\pi,\;\mbox{for all}\; n \geq 0, $$ which is an interesting representation for $\pi$. 
Also, for this case, it is shown later (see Remark 3.2) that \eqref{eqn2.5} reduces to 	\eqref{eqn1.1}.

\item[(iii)] When $p_1=x$, $p_2=y$ and $n=1$, we get, 
\begin{equation*}
B(x,y)=B(x,y+1)+B(x+1,y),  
\end{equation*}
the basic identity  in \eqref{eqn2.3}.  
Thus, Theorem 2.1 extends both the identities in \eqref{eqn1.1} and in \eqref{eqn2.3}.
\end{enumerate}
}
\end{remarks}
\noindent Our next result is an interesting identity that relates binomial coefficients and beta functions on one side and to a simple expression on the other side. The proof relies on the binomial inversion formula (see, for example, Aigner (2007)) which we include here for ease of reference. \\ 

\noindent {\bf The Binomial Inversion Formula.} For a positive sequence of $\{a_n\}_{n \geq 0 }$ of real numbers, define the real sequence $\{b_n\}_{n \geq 0 }$ by $$ b_n=\sum_{j=0}^{n}(-1)^{j}\binom{n}{j}a_j .$$  Then, for all $n \geq 0$, the binomial inversion of $b_n$ is given by  $$ a_n=\sum_{j=0}^{n}(-1)^{j}\binom{n}{j}b_j .$$ 

\begin{theorem} Let $s > 0$ and $n \geq 0$ be an integer. Then, 
\begin{eqnarray}\label{eqn2.8}
\sum_{j=0}^{n}(-1)^{j}\binom{n}{j}B(j+1,s)= \frac{1}{s+n}.
\end{eqnarray}
\end{theorem}

%\end{remark}
\begin{proof}
\noindent The following binomial identity is known:
\begin{eqnarray}\label{eqn2.9}
\sum_{j=0}^{n}(-1)^{j}\binom{n}{j}\left(\frac{s}{s+j}\right)=\prod_{j=1}^{n} \left(\frac{j}{s+j}\right).
\end{eqnarray}
Recently, Peterson (2013) and  Vellaisamy (2015) gave a probabilistic proof of the above identity.  Note that the right hand side of \eqref{eqn2.9} can also be written as 
\begin{eqnarray*}
\prod_{j=1}^{n} \left(\frac{j}{s+j}\right)
= \frac{\Gamma(n+1)\Gamma(s+1)}{\Gamma(s+n+1)}= \frac{\Gamma(n+1)\; s\Gamma(s)}{\Gamma(s+n+1)}=s B(n+1,s). 
\end{eqnarray*}
Then, the identity in \eqref{eqn2.8} becomes
\begin{eqnarray}\label{eqn2.10}
\sum_{j=0}^{n}(-1)^{j}\binom{n}{j}\left(\frac{1}{s+j}\right)= B(n+1,s). 
\end{eqnarray}
Applying  the binomial inversion formula to \eqref{eqn2.10} with $a_j=\frac{1}{s+j}$ and $b_n=B(n+1,s)$, we get $$\sum_{j=0}^{n}(-1)^{j}\binom{n}{j}  B(j+1,s)= \frac{1}
 {s+n},$$ which proves the result.     
\end{proof}

\begin{remark} {\em When $n=1$, equation \eqref{eqn2.8} becomes $$B(1,s)-B(2,s)=\frac{1}{s+1}=B(1,s+1),$$ which coincides with the basic beta-function identity in \eqref{eqn2.3}, when $x=1$. Thus, \eqref{eqn2.8} can be viewed as another generalization of \eqref{eqn2.3} in the case when $x$ is a positive integer. Also, when $n=2$, $$B(1,s)-2 B(2,s)+B(3,s)=\frac{1}{s+2},$$ which can be verified using the basic identity in \eqref{eqn2.3}. It may be of interest to provide a different proof of Theorem 2.2 based on induction or combinatorial arguments. }
\end{remark}

\noindent It is easy to see that the derivative of the beta function is  
\begin{equation} \label{eqn2.11}
\frac{\partial}{\partial y} B(x,y)=B(x,y) \Big( \psi(y)-\psi(x+y)\Big),
\end{equation}
 where $\psi(x)=\frac{\Gamma'(x)}{\Gamma(x)}$ is the digamma function. The following result is a consequence of Theorem 2.2. 
 
%\begin{equation}
%A_n=\sum_{k=1}^{n}\sum_{\substack{\sum\limits_{i=1}^{m}\sum\limits_{j=1}^{n}x_{ij}=k;%\\\sum\limits_{i=1}^{m}\left(\sum\limits_{j=1}^{n}jx_{ij}\right)=n}}f^{(x_{1*},\ldots,x_{m*})}\frac{u_{11}^{x_{11}}\ldots %u_{1n}^{x_{1n}}}{x_{11}!\ldots x_{1n}!}\ldots\frac{u_{m1}^{x_{m1}}\ldots u_{mn}^{x_{mn}}}{x_{m1}!\ldots %x_{mn}!}\label{A4-parameter}
%\end{equation}

 %$$\sum_{j=0}^{n}(-1)^{j}\binom{n}{j}  B(j+1,s)= \frac{1}{s+n}.$$ 

\begin{theorem}  For $s > 0$ and an integer $n \geq 0$, 
\begin{eqnarray}\label{eqn2.12}
\sum_{j=0}^{n}\sum_{i=0}^j (-1)^{j}\binom{n}{j} \frac{B(j+1,s)}{s+i}= \frac{1}{(s+n)^2}.
\end{eqnarray}
\end{theorem}

\begin{proof}The proof proceeds by taking the derivative of both sides of $(2.7)$ with respect to $s$. From the right-hand side, we get, 
\begin{equation}\label{eqn2.13}
\frac{\partial}{\partial s} \Big(\frac{1}{s+n}\Big)= \frac{-1}{(s+n)^2}.
\end{equation}
 Also, form the left-hand side,
\begin{align}
\frac{\partial}{\partial s}\sum_{j=0}^{n}(-1)^{j}\binom{n}{j}  B(j+1,s) =& \sum_{j=0}^{n}(-1)^{j}\binom{n}{j}\frac{\partial}{\partial s} B(j+1,s)\nonumber\\=& \sum_{j=0}^{n}(-1)^{j}\binom{n}{j}  B(j+1,s) \Big( \psi(s)-\psi(j+1+s)\Big), \label{eqn2.14}
\end{align}
using \eqref{eqn2.11}. Further, it is known that the digamma function $ \psi(x) $satisfies 
\begin{equation} \label{eqn2.15}
\psi(x+1)-\psi(x)=\frac{1}{x}. 
\end{equation}
Using \eqref{eqn2.15} iteratively leads to 
\begin{equation} \label{eqn2.16}
\psi(x+j+1)-\psi(x)=\displaystyle \sum_{i=0}^j \frac{1}{x+i},
\end{equation}
for a nonnegative integer $j$.

\noindent Using \eqref{eqn2.13} \eqref{eqn2.14} and \eqref{eqn2.16}, we get
\begin{eqnarray*}
\frac{-1}{(s+n)^2}&=&\sum_{j=0}^{n}(-1)^{j}\binom{n}{j}  B(j+1,s) \Big( \psi(s)-\psi(j+1+s)\Big)\\&=&\sum_{j=0}^{n} \sum_{i=0}^j (-1)^{j}\binom{n}{j}  B(j+1,s) (-1) \frac{1}{s+i},
\end{eqnarray*}
which proves the result.  
\end{proof}

Our aim next is to extend the identity in \eqref{eqn1.3}.

\begin{theorem}
Let $p>0$ and and $n$ be a positive integer. Then	
\begin{equation}\label{eqn2.17}
 \sum_{k=0}^{n} (-1)^k \binom{n}{k} B( p+k, p+n-k)=
\begin{cases*}
\displaystyle{	\frac {n! \Gamma(p) \Gamma(p+\frac{n}{2})} {\Gamma(\frac{n}{2}+1) \Gamma(2p+n)}}, & if $n$ is even \\
  \phantom{0}0,  & if $n$ is odd.
\end{cases*}	
\end{equation}	
\end{theorem}

\begin{proof}To prove the result, we consider the rv $X= X_1-X_2$, where $X_1$ and $X_2$ are as before independent gamma $G(p)$ variables.
Note first that since $X_1$ and $X_2$ are independent and identically distributed, we have   $X_1-X_2 \stackrel{d}{=}  X_2-X_1$. Here $\stackrel{d}{=}$ denotes the equality in distributions. That is, $X$ and $-X$ have the same distributions on $\mathbb{R}$, which implies the 
density of $X$ is symmetric about zero. Hence, $E(X^n)=0$ if $n$ is an odd integer.

\noindent Next, we compute the even moments of $X$. The earlier approach of finding the moments of $X$ using the probability density function is tedious. This is because the  density of $X$ is very complicated and it involves Whittaker's W-function (see Mathai (1993)). Therefore, we use the moment generating function ($MGF$) approach to find the moments of $X$.

\noindent It is known (see Rohatgi and Saleh (2002, p.~212) that the  $MGF$ of $X_1$ is $M_{X_1}(t)= E(e^{tX_1})= (1-t)^{-p}.$ Hence, the $MGF$ of $X$ is
 \begin{align}
  M_{X}(t) =& M_{X_1-X_2}(t)=M_{X_1}(t)M_{X_2}(-t) \nonumber \\
            =& (1-t)^{-p} (1+t)^{-p} \nonumber \\
            =& (1-t^2)^{-p}        
\end{align}
which exits for $|t|<1.$	

\noindent Using the result, for $p>0$ and  |q|<1, that
\begin{equation*}
(1-q)^{-p}= \sum_{n=0}^{\infty} \frac{\Gamma(n+p) q^n}{\Gamma(n+1) \Gamma(p)},
\end{equation*}
we have
\begin{equation} \label{eqn2.19}
M_{X}(t)= (1-t^2)^{-p}= \sum_{n=0}^{\infty} \frac{\Gamma(n+p) t^{2n}}{\Gamma(n+1) \Gamma(p)}.
\end{equation}
Hence, for  $n \geq 1$, we have from \eqref{eqn2.19}
\begin{equation} \label{eqn2.20}
E(X^{2n})=M_{X}^{(2n)}(t)|_{t=0}=  \frac{\Gamma(n+p) {(2n)!}}{\Gamma(n+1) \Gamma(p)},
\end{equation}	
where $f^{(k)}$ denotes the $k$-th derivative of $f$. Thus,
we have shown that	
\begin{equation} \label{eqn2.21}
E(X^{n})=
\begin{cases*}
\displaystyle  \frac{n! \Gamma(\frac{n}{2}+p) }{\Gamma(\frac{n}{2}+1) \Gamma(p)}, & if $n$ is even\\
\phantom{0}0,  & if $n$ is odd.
\end{cases*}
\end{equation}
Next we compute the moments of $X$, using the binomial theorem. Note that
\begin{align}\label{eqn2.22}
E(X^n) =& E(X_1-X_2)^n= \sum_{k=0}^{n} (-1)^k \binom{n}{k} E(X_1^k) E(X_2^{n-k}) \nonumber \\
=& \sum_{k=0}^{n}  (-1)^k \binom{n}{k} \frac{\Gamma(p+k)}{\Gamma(p)}\frac{\Gamma(p+n-k)}{\Gamma(p)}.    \end{align}
Equating \eqref{eqn2.21} and \eqref{eqn2.22}, we get
\begin{align}\label{eqn2.23}
\sum_{k=0}^{n}  (-1)^k \binom{n}{k} {\Gamma(p+k)}{\Gamma(p+n-k)}=
\begin{cases*}
 \displaystyle \frac{n!\Gamma(\frac{n}{2}+p) \Gamma(p)}{\Gamma(\frac{n}{2}+1)}, & if $n$ is even \\ 
 \phantom{0}0,  & if $n$ is odd. 
\end{cases*}
\end{align}
which is an interesting identity involving gamma functions and 
binomial coefficients.
Dividing now  both sides of \eqref{eqn2.23}  by $\Gamma(2p+n)$, the result follows.
\end{proof}

\noindent We next show that the identitiy in \eqref{eqn1.2} follows as a special case.
\begin{corollary} Let $n$ be a positive integer. Then 
\begin{equation}\label{eqn2.24}
\sum_{k=0}^{n} (-1)^k \binom{2k}{k}  \binom{2n-2k}{n-k} =
\begin{cases*}
2^n \binom{n}{\frac{n}{2}}, & if  $n$ is even  \\
\phantom{0 }0, & if  $n$ is odd.
\end{cases*} 
\end{equation}
\end{corollary}
 
\begin{proof} Let $p=\frac{1}{2}$ in \eqref{eqn2.23} and it suffices to consider  the case $n$ is even. Then
\begin{align*}
\sum_{k=0}^{n}  (-1)^k \binom{n}{k} {\Gamma(k+\frac{1}{2})}{\Gamma(n-k+\frac{1}{2})}=
\displaystyle \frac{{n!}\Gamma(\frac{n}{2}+\frac{1}{2}) \Gamma(\frac{1}{2})}{\Gamma(\frac{n}{2}+1)}
\end{align*}
That is,
\begin{align} \label{eqn2.25}
\sum_{k=0}^{n}  (-1)^k \binom{n}{k} \frac{\Gamma(k+\frac{1}{2})}{\Gamma(\frac{1}{2})}\frac{\Gamma(n-k+\frac{1}{2})}{\Gamma(\frac{1}{2})}  =
\displaystyle \frac{{n!}\Gamma(\frac{n}{2}+\frac{1}{2}) }{{\Gamma(\frac{1}{2})}\Gamma(\frac{n}{2}+1)}.
\end{align}
Note  that, 
\begin{eqnarray}
\frac{\Gamma(n+\frac{1}{2})}{\Gamma(\frac{1}{2})}&=&\frac{\Big(n-\frac{1}{2}\Big) \Big(n-\frac{3}{2}\Big) \cdots \Big(\frac{3}{2}\Big) \Big(\frac{1}{2}\Big) \Gamma \Big(\frac{1}{2} \Big)}{\Gamma \Big(\frac{1}{2} \Big)}\nonumber\\&=& 
\frac{(2n-1) \cdot (2n-3) \cdots 3 \cdot 1}{2^n} \nonumber \\ &=& \frac{(2n)!}{ n! 4^n}. \label{eqn2.26}
\end{eqnarray}

\noindent Using \eqref{eqn2.26} in \eqref{eqn2.25}, we get
\begin{align*} 
\displaystyle \sum_{k=0}^{n}  (-1)^k  \frac{n!}{k! (n-k)!}\frac{ (2k)!}{4^k k!} \frac{(2n-2k)!}{4^{(n-k)} (n-k)! } =& \displaystyle
 \frac{n! n!}{4^{\frac{n}{2}} (\frac{n}{2})!(\frac{n}{2})!} 
\end{align*}
That is,
\begin{equation}\label{eqn2.27}
\sum_{k=0}^{n} (-1)^k \binom{2k}{k}  \binom{2n-2k}{n-k} =
\displaystyle \frac{n! 4^n }{(\frac{n}{2})!(\frac{n}{2})! 4^{\frac{n}{2}}} \nonumber \\
  = \displaystyle 2^n \binom{n}{\frac{n}{2}},
\end{equation}
 which proves the result.
\end{proof}

\section{An Extension}

In this section, we discuss an extension of the beta-function identity given in Theorem 2.1. This result in particular extends the main result of Chang and Xu (2011). Let $p_i > 0,\;1 \leq i \leq m,$ and consider the beta function of $m$-variables defined by 
\begin{equation}
B(p_1,\cdots,p_m)= \int_\T x_1^{p_1-1}x_2^{p_2-1} \cdots x_m^{p_m-1} (1-p_1-p_2-\cdots p_m)^{p_m-1}dx_1 \cdots dx_m,
\end{equation} 
where $\T = (0,1) \times \cdots \times (0,1)$. It is well known that $B(p_1,\cdots,p_m)$ can be expressed as a ratio of gamma functions as
\begin{equation}
B(p_1,\cdots,p_m)= \frac{\Gamma(p_1)\Gamma(p_2)\cdots\Gamma(p_m)}{\Gamma(p_1+\cdots+p_m)}.
\end{equation}

\begin{theorem} Let $p_1, \cdots, p_m \geq 0.$ Then for any non-negative integer $n$, 
\begin{equation}\label{eqn3.3}
\sum_{\substack{k_j\geq 0;~ 1 \leq j \leq n; \\ \sum_{j=1}^m  k_j=n}} \binom {n} {k_1,\cdots, k_m} B(p_1+k_1,\cdots,p_m+k_m)=B(p_1,\cdots,p_m)
\end{equation}
where $\binom{n}{k_1,\cdots,k_m}= \frac{n!}{k_1! \cdots k_m!}$ denotes the multinomial coefficient. 
\end{theorem} 

\begin{proof}Let $X_1,\cdots X_m$ be $m$ independent gamma random variables, where $X_i \sim G(p_i),$ 
$1 \leq i \leq m$. Then it is known that 
$$Y= \sum_{i=1}^m X_i \sim G(p_1+\cdots+p_m).$$ 
Also, from \eqref{eqn2.4}, 
\begin{equation} \label{eqn3.4}
 E(Y^n)=\frac{\Gamma(p_1 + p_2+ \cdots p_m+ n)}{\Gamma(p_1+\cdots+p_m)}.
\end{equation}
 Since $X_i$'s  are independent, we have by multinomial theorem, 
\begin{eqnarray}
E(X_1+\cdots+X_m)^n&=& E \Bigg[\sum_{\substack{k_j \geq 0,\;1 \leq j \leq m\\ \sum_{J=1}^m  k_j=n}}  \binom {n} {k_1,\cdots, k_m} X_1^{k_1}\cdots X_m^{k_m}\Bigg] \nonumber
 \\&=& \sum_{\substack{k_j \geq 0, \;1 \leq j \leq m\\ \sum_{J=1}^m  k_j=n}}\binom {n} {k_1,\cdots, k_m} E(X_1^{k_1}) \cdots E(X_m^{k_m})\nonumber\\
 &=& \sum_{\substack{x_j \geq 0,\;1 \leq j \leq m\\ \sum_{J=1}^m  k_j=n}} \binom {n} {k_1,\cdots, k_m}
\frac{\Gamma(p_1+k_1)\cdots \Gamma(p_m+k_m)}{\Gamma (p_1)\cdots \Gamma(p_m)}. \label{eqn3.5}  
\end{eqnarray} 
Equating \eqref{eqn3.4} and \eqref{eqn3.5}, we obtain 
$$ \sum_{\substack{k_j \geq 0,\; 1 \leq j \leq m\\ \sum_{J=1}^m  k_j=n}}\binom {n} {k_1,\cdots, k_m}
\frac{\Gamma(p_1+k_1)\cdots \Gamma(p_m+k_m)}{\Gamma (p_1)\cdots \Gamma(p_m)}=\frac{\Gamma(p_1 + p_2+ \cdots p_m+ n)}{\Gamma(p_1+\cdots+p_m)}.$$
That is, 
$$ \sum_{\substack{k_j \geq 0,\; 1 \leq j \leq m\\ \sum_{J=1}^m  k_j=n}}\binom {n} {k_1,\cdots, k_m}
\frac{\Gamma(p_1+k_1)\cdots \Gamma(p_m+k_m)}{\Gamma(p_1 + p_2+ \cdots p_m+ n)}=\frac{\Gamma (p_1)\cdots \Gamma(p_m)}{\Gamma(p_1+\cdots+p_m)},$$
from which the result follows.  
\end{proof}

\begin{remark} Obviously, when $m=2,$ the identity in \eqref{eqn3.3} reduces to 
\begin{equation*}
\displaystyle \sum_{\substack{k_j\geq 0;~ 1 \leq j \leq 2; \\ k_1+k_2=n}} \binom {n} {k_1, k_2} B(p_1+k_1, p_2+k_2)= \sum_{k=0}^n \binom{n}{k}B(p_1+k,p_2+n-k)=B(p_1,p_2),
\end{equation*}	
which is \eqref{eqn2.5}.
\end{remark}

\noindent Our next result shows that the identity in \eqref{eqn1.2} follows as a special case.

\begin{corollary}
When $p_1=p_2=\cdots=p_m=\frac{1}{2}$, the identity in \eqref{eqn3.3} reduces to
\begin{equation}\label{eqn3.6}
 \sum_{\substack{k_j \geq 0,\; 1 \leq j \leq m\\ \sum_{J=1}^m  k_j=n}} \binom{2k_1}{k_1}\cdots\binom{2k_m}{k_m}=\frac{4^n}{n!} \frac{\Gamma(n+\frac{m}{2})}{\Gamma(\frac{m}{2})},
 \end{equation}
  for all integers $m,n \geq 1$. 
\end{corollary}

\begin{proof}
Putting $p_1=p_2=\cdots p_m=\frac{1}{2}$ in \eqref{eqn3.3}, we obtain, 
$$\sum_{\substack{k_j \geq 0,\; 1 \leq j \leq m\\ \sum k_j=n}} \binom {n} {k_1,\cdots, k_m} B\Big(\frac{1}{2}+k_1,\cdots, \frac{1}{2}+k_m\Big)=B\Big(\frac{1}{2},\cdots, \frac{1}{2}\Big)$$
This implies, 
\begin{eqnarray*}
 \sum_{\substack{k_j \geq 0,\; 1 \leq j \leq m\\ \sum k_j=n}} \binom {n} {k_1,\cdots, k_m} \frac{\Gamma(\Big(\frac{1}{2}+k_1\Big),\cdots, \Gamma \Big(\frac{1}{2}+k_m\Big)}{\Gamma \Big(n+\frac{m}{2}\Big)}&=&\frac{\Gamma \Big(\frac{1}{2}\Big) \cdots \Gamma \Big(\frac{1}{2}\Big)}{\Gamma \Big(\frac{m}{2}\Big)},
\end{eqnarray*}
or, equivalently, 
\begin{eqnarray*}
 \sum_{\substack{k_j \geq 0,\; 1 \leq j \leq m\\ \sum k_j=n}} \binom {n} {k_1,\cdots, k_m} \frac{\Gamma(\Big(\frac{1}{2}+k_1\Big),\cdots, \Gamma \Big(\frac{1}{2}+k_m\Big)}{\Gamma \Big(\frac{1}{2}\Big) \cdots \Gamma \Big(\frac{1}{2}\Big)}=\frac{\Gamma \Big(n+\frac{m}{2}\Big)}{\Gamma \Big(\frac{m}{2}\Big)}.
\end{eqnarray*}
Using \eqref{eqn2.26}, we get 
\begin{eqnarray} \label{eqn3.7}
 \sum_{\substack{k_j \geq 0,\; 1 \leq j \leq m\\ \sum k_j=n}} \binom {n} {k_1,\cdots, k_m} \frac{(2k_1)!,\cdots, (2k_m)!}{4^{k_1}(k_1)! \cdots 4^{k_m}(k_m)!} 
=\frac{\Gamma \Big(n+\frac{m}{2}\Big)}{\Gamma \Big(\frac{m}{2}\Big)}.
\end{eqnarray}
We can write \eqref{eqn3.7} as 
\begin{eqnarray}
 \sum_{\substack{k_j \geq 0,\; 1 \leq j \leq m\\ \sum k_j=n}} \binom {2k_1} {k_1} \binom {2k_2} {k_2} \cdots \binom {2k_m} {k_m}=\frac{4^n}{n!}\Bigg(\frac{\Gamma \Big(n+\frac{m}{2}\Big)}{\Gamma \Big(\frac{m}{2}\Big)}\Bigg),
\end{eqnarray}
which proves the corollary. 
\end{proof}

\begin{remark} \label{rem3.2} Note that  when $m=2$, \eqref{eqn3.6} reduces tom \eqref{eqn1.1}. This implies also that when $p=\frac{1}{2}$, \eqref{eqn2.5} reduces to 	\eqref{eqn1.1}.
\end{remark}

\begin{remark} {\em
\begin{enumerate}
\item[(i)] Let $m$ be even so that $m=2l$ for some positive integer $l$. Then the right hand side of \eqref{eqn3.6} is 
$$ \frac{4^n}{n!}\frac{\Gamma (n+l)}{\Gamma (l)}=4^n \binom{n+l-1}{n}=4^n \binom{n+\frac{m}{2}-1}{n}.$$ 

Similarly, when $m$ is odd, say $m=2l+1$, 
\begin{eqnarray*}
\frac{4^n}{n!}\frac{\Gamma (n+\frac{m}{2})}{\Gamma (\frac{m}{2})}&=& \frac{4^n}{n!}\frac{\Gamma (n+l+\frac{1}{2})}{\Gamma (l+\frac{1}{2})}=\frac{4^n}{n!}\Bigg(\frac{\Gamma (n+l+\frac{1}{2})}{\Gamma(\frac{1}{2})}\Bigg)\Bigg(\frac{\Gamma(\frac{1}{2})}{\Gamma (l+\frac{1}{2})}\Bigg)\\
&=& \binom{2n+2l}{2n}\Bigg(\frac{(2 n)!}{n!\; n!}\Bigg) \Bigg(\frac{l!\; n!}{(n+ l)!}\Bigg)\; (\mbox{using}\; \eqref{eqn2.6})\\
&=&\frac{\binom{2n+2l}{2n} \binom{2n}{n}}{\binom{n+l}{n}}\\
&=&\frac{\binom{2n+m-1}{2n} \binom{2n}{n}}{\binom{n+\frac{m-1}{2}}{n}}, 
\end{eqnarray*}
since $2l={m-1}$. Thus, we have from \eqref{eqn3.6},

\begin{equation} \label{eqn3.9}
%\begin{displaymath}
   \sum_{\substack{k_j \geq 0,\; 1 \leq j \leq m\\ \sum k_j=n}} \binom {2k_1} {k_1} \binom {2k_1} {k_1} \binom {2k_2} {k_2} \cdots \binom {2k_m} {k_m}= \left\{
     \begin{array}{lr} \displaystyle
       4^n \binom{n+\frac{m}{2}-1}{n}, &  \mbox{if}\; m\;\mbox{is even}\\
       & \\
      \displaystyle \frac{\binom{2n+m-1}{2n} \binom{2n}{n}}{\binom{n+\frac{m-1}{2}}{n}}, &  \mbox{if}\; m\;\mbox{is odd},
     \end{array}
   \right.
%\end{displaymath}
\end{equation}
which is equation $(3)$ of Miki{\' c} (2016). Indeed, Miki{\' c} (2016)
provided a combinatorial proof of the above result based on recurrence relations.
\end{enumerate}
}
\end{remark}

\noindent {\bf Acknowledgments}. This work was completed while the first author was visiting  the Department of Statistics and Probability, Michigan State University, during Summer-2017. We are grateful to Professor Frederi Viens for his support and encouragements. 
\vone
\noindent{\bf{\Large References}} 

\noindent Aigner, M. (2007). {\it A Course in Enumeration}. Berlin, Springer-Verlag.

\noindent Bagdasaryan, A. (2015).  A combinatorial identity involving gamma function and Pochhammer symbol {\it Appl. Math. Comput.}, {\bf 256}, 125-130. 

\noindent Chang, G. and  Xu, C. (2011). Generalization and probabilistic proof of a combinatorial identity. {\it Amer. Math. Monthly}, {\bf 118}, 175-177.

\noindent De Angelis, V. (2006).  Pairings and signed permutations. {\it Amer. Math. Monthly}, {\bf 113},
642-644.

%\noindent Holst, L. (2013). Probabilistic Proofs of Euler identities. {\it Journal of Applied Probability}, {\it 50}, 1206-1212.

\noindent Mathai, A.M. (1993). On non-central generalized Laplacianness of quadratic forms in
normal variables. {\it J. Multivariate Anal.}, {\bf 45}, 239-246.

\noindent Mickic{\' c}, J (2016).  A proof of a famous identity concerning the convolution of the central binomial coefficients. {\it J. Integer Seq.}, {\bf 19}, Article 16.6.6. 
%\noindent Pace, L. (2011). Probabilistically proving $\zeta(2)=
%\frac{\pi^2}{6}$. {\it American Mathematical Monthly}, 118, 641-643.

\noindent Nagy, G. V. (2012). A combinatorial proof of Shapiro's Catalan convolution. {\it Adv. in Appl. Math.}, {\bf 49}, 391-396.

 \noindent Pathak, A. K. (2017). A simple probabilistic proof for the alternating convolution of the central binomial coefficients. To appear in {\it  Amer. Statist.}

\noindent Peterson, J. (2013). A probabilistic proof of a binomial identity. {\it Amer. Math. Monthly}, {\bf 120}, 558-562.

%\noindent Pakshirajan, R. P. (2013). {\it Probability %Theory:A foundational Course}. Hindustan Book Agency, %New Delhi, India.

\noindent Petkovsek, M, Wilf, H and Zeilberger, D. (1996).  $A=B$. {\it  https://www.math.upenn.edu/wilf \\
	/AeqB.html.}

\noindent Riordan, J. (1968). {\it Combinatorial Identities}. Wiley, New York.
 
 \noindent Rohatgi, V. K. and Saleh. A. K. Md. E. (2002). {\it An Introduction to Probability and Statistics}. Wiley, New York.
 
\noindent Spivey, M. Z. (2014). A combinatorial proof for the alternating convolution of the central binomial
coefficients. {\it Amer. Math. Monthly}, {\bf 121}, 537-540. 
 
\noindent Srinivasan, G. K. (2007). The gamma function: an eclectic tour.  {\it Amer. Math. Monthly}, {\bf 114}, 297-316.

\noindent Stanley, R. (1997). {\it Enumerative Combinatorics}, Vol. 1, Cambridge Studies in Advanced Mathematics, {\bf 49}, Cambridge University Press.

\noindent Sved, M.(1984).  Counting and recounting: the aftermath, {\it Math. Intelligencer}, {\bf 6}, 44-46.

\noindent Vellaisamy, P. and A. Zeleke, A. (2017). Exponential order statistics, the Basel problem and
combinatorial identities. Submitted.

\noindent Vellaisamy, P. (2015).  On probabilistic proofs of certain binomial identities. {\it Amer. Statist.}, {\bf 69}, 241-243.

\vone
\end{document}